\documentclass[12pt]{article}

\usepackage[margin=2cm]{geometry}
\usepackage{amsmath,amssymb,amsthm,mathtools}
\usepackage[hidelinks]{hyperref}

\title{Green--Wasserstein Inequality on Compact Surfaces\footnote{To appear in \textit{Journal of Inequalities and Applications.}}}
\author{Maja Gwóźdź\\University of Zurich \& ETH Z\"{u}rich, Zurich, Switzerland\\\texttt{mgwozdz@ethz.ch}}
\date{}

\theoremstyle{plain}
\newtheorem{theorem}{Theorem}

\newtheorem{lemma}{Lemma}

\theoremstyle{definition}
\newtheorem{definition}{Definition}
\newtheorem{problem}{Problem}
\newtheorem{remark}{Remark}

\newcommand{\M}{M}
\newcommand{\dd}{\,\mathrm{d}}
\newcommand{\W}{W_2}
\newcommand{\dist}{\mathrm{d}_g}
\newcommand{\E}{\mathbb{E}}

\begin{document}
\maketitle

\begin{abstract}
Let $(\M,g)$ be a compact connected two-dimensional Riemannian manifold without boundary. In this note, we answer a question posed by Steinerberger: can one remove the $\sqrt{\log n}$ factor in the two-dimensional Green--Wasserstein inequality while keeping the unrenormalized off-diagonal Green term? We show that this is impossible on any compact connected surface: there is no inequality of the same form that holds uniformly over point sets with an $O(n^{-1/2})$ remainder for all $n$. We argue by contradiction and combine a second-moment estimate for the random Green energy of i.i.d. samples with the semi-discrete random matching asymptotics of Ambrosio--Glaudo.
\end{abstract}

\section{Introduction}

Let $(\M,g)$ be a compact connected two-dimensional Riemannian manifold without boundary, $\dist$ is the induced distance, and let $G(x,y)$ denote the symmetric mean-zero Green function of the Laplacian (Definition \ref{def:green}). We work with the normalized volume measure $dx=\mathrm{vol}(\M)^{-1}\dd\mathrm{vol}$. Whenever we integrate in the variable $y$, we write $dy$ for the same measure. For points $x_1,\dots,x_n\in \M$, we define the empirical measure:
\[
\mu_n \coloneqq \frac1n\sum_{i=1}^n \delta_{x_i}.
\]
Steinerberger \cite{Ste21} showed that for $d\ge 3$, the following Green--Wasserstein inequality holds:
\[
\W(\mu_n,dx)\ \lesssim_{\M}\ n^{-1/d} + \frac1n\left|\sum_{i\neq j} G(x_i,x_j)\right|^{1/2}.
\]
In \cite[Problem 53]{SteOpen}, he gives the following estimate:
\begin{equation}\label{eq:Steinerberger2D}
\W\!\left(\frac1n\sum_{k=1}^n \delta_{x_k},\,dx\right)
\lesssim_\M
\frac1n\left|\sum_{i\neq j} G(x_i,x_j)\right|^{1/2}
+
\begin{cases}
\sqrt{\frac{\log n}{n}} & \text{if } d=2,\\
n^{-1/d} & \text{if } d\ge 3.
\end{cases}
\end{equation}
In this note, we use
\[
\sum_{i\neq j}(\cdots)\ \coloneqq\ \sum_{\substack{1\le i,j\le n\\ i\neq j}}(\cdots)
\]
to denote the ordered sum over distinct indices. Notice that, by symmetry of $G$, it is twice the sum over $i<j$. Steinerberger asked whether the factor $\sqrt{\log n}$ in \eqref{eq:Steinerberger2D} can be removed while keeping the same Green-energy term. The precise question is as follows.

\begin{problem}[Steinerberger {\cite[Problem 53]{SteOpen}}]\label{prob:Steinerberger}
In dimension $d=2$, can one replace the remainder $\sqrt{\frac{\log n}{n}}$ in \eqref{eq:Steinerberger2D}
by $\frac{1}{\sqrt n}$ while keeping the same off-diagonal Green term? In other terms, does there exist
$C_\M>0$ such that for all $n\in\mathbb N$ and all $x_1,\dots,x_n\in \M$,
\[
\W\!\left(\frac1n\sum_{k=1}^n \delta_{x_k},\,dx\right)
\le C_\M\left(\frac1{\sqrt n}
+\frac1n\left|\sum_{i\neq j} G(x_i,x_j)\right|^{1/2}\right)?
\]
\end{problem}
\noindent Our main result, Theorem \ref{thm:main}, gives a negative answer to Problem \ref{prob:Steinerberger}.

\begin{remark}[The diagonal]\label{rem:diagonal}
We note that the Green function $G$ is smooth on $(\M\times \M)\setminus \{(x,x):x\in\M\}$ with a logarithmic singularity along the diagonal.
In a deterministic context, if $x_i=x_j$ for some $i\neq j$, then the term $G(x_i,x_j)$ is not finite. It is then natural to interpret the Green term in \eqref{eq:Steinerberger2D} and Problem \ref{prob:Steinerberger} as $+\infty$, in which case the inequality holds trivially.
In the random settings that we use in this note, collisions occur with probability $0$ because $dx$ is non-atomic.
\end{remark}

We now recall some definitions that will be useful in the main argument.
For Borel probability measures $\mu,\nu$ on $\M$, let us denote by $\Gamma(\mu,\nu)$ the set of couplings of $\mu$ and $\nu$ on $\M\times \M$.
The quadratic Wasserstein distance is
\[
\W(\mu,\nu)\coloneqq \left(\inf_{\gamma\in \Gamma(\mu,\nu)} \int_{\M\times \M} \dist(x,y)^2\,\dd\gamma(x,y)\right)^{1/2}.
\]
Given that $\M$ is compact, it follows that $\W(\mu,\nu)$ is always finite and satisfies $\W(\mu,\nu)\le \mathrm{diam}(\M)$. Let $\Delta$ be the Laplace--Beltrami operator on $(\M,g)$, which is realized here as a self-adjoint operator on $L^2(\M,dx)$ with domain $H^2(\M)$. Again, since $\M$ is compact and connected, $\ker(-\Delta)$ consists of the constant functions.

\begin{definition}[Mean-zero Green function]\label{def:green}
A mean-zero Green function for $-\Delta$ is a measurable function
\[
G:\M\times \M \to \mathbb{R}\cup\{+\infty\}
\]
with the following properties.
\begin{enumerate}
\item For each $x\in \M$, $y\mapsto G(x,y)$ is locally integrable on $\M$ and satisfies the normalization:
\begin{equation}\label{eq:GreenMeanZero}
\int_{\M} G(x,y)\,dy=0.
\end{equation}
\item For each $f\in C^\infty(\M)$ with $\int_{\M} f\,dx=0$, the function
\[
u(x)\coloneqq \int_{\M} G(x,y)\, f(y)\,dy
\]
is a weak solution of $-\Delta u = f$ with $\int_{\M}u\,dx=0$.
\end{enumerate}
\end{definition}

Recall that the uniqueness and existence of the mean-zero Green function, as a distribution kernel, are guaranteed. We can also choose it to be symmetric (for example, see \cite{Aubin98} or \cite{Rosenberg97}). We now fix a symmetric mean-zero Green function $G$, so that $G(x,y)=G(y,x)$ holds.
By symmetry and \eqref{eq:GreenMeanZero}, we obtain:
\[
\int_{\M}G(y,x)\,dy=0\qquad\text{for every }x\in \M.
\]

\begin{remark}[Local singularity]\label{rem:logsing}
In dimension $2$, the following classical local expansion (in geodesic normal coordinates) is true:
\[
G(x,y)= -\frac{1}{2\pi}\log \dist(x,y) + H(x,y).
\]
Note that $H$ extends continuously to the diagonal and is smooth off the diagonal.
In particular, we have $G(x,\cdot)\in L^1(\M,dx)$ for each fixed $x$.
\end{remark}

\begin{lemma}\label{lem:L2Green}
Let $(\M,g)$ be a compact connected $2$-dimensional Riemannian manifold without boundary, and let $G$ be the symmetric mean-zero Green function of $-\Delta$.
It then holds that $G\in L^2(\M\times \M, dx\otimes dx)$, and so
\[
\sigma^2 \coloneqq \int_{\M}\int_{\M} G(x,y)^2\,dx\,dy <\infty.
\]
\end{lemma}

\begin{proof}
Recall that $\mathrm{vol}$ denotes the Riemannian volume measure on $(\M,g)$ and $dx = \mathrm{vol}(\M)^{-1}\dd\mathrm{vol}$. Let $\mathrm{inj}(\M)>0$ be the injectivity radius of $(\M,g)$ and set $r_0\coloneqq \mathrm{inj}(\M)/4$. We now apply standard results on the singularity structure of the Green function in dimension $2$ (for instance, see \cite{Aubin98}). In particular, the function
\[
H(x,y)\coloneqq G(x,y)+\frac{1}{2\pi}\log \dist(x,y),
\qquad x\neq y,
\]
extends continuously to the set $\{(x,y)\in \M\times\M: \dist(x,y)\le 2r_0\}$ (this also holds across the diagonal). Observe that this set is compact, so there exists $C_0>0$ such that
\[
|H(x,y)|\le C_0\qquad\text{whenever }\dist(x,y)\le 2r_0.
\]
As a result, for all $x\neq y$ with $\dist(x,y)<2r_0$, it follows that
\begin{equation}\label{eq:GreenLogBound}
|G(x,y)| \le \frac{1}{2\pi}|\log \dist(x,y)|+C_0 \le C\bigl(1+|\log \dist(x,y)|\bigr)
\end{equation}
for some constant $C>0$.

We now split $\M\times\M$ into the near-diagonal region
\[
D\coloneqq \{(x,y)\in\M\times\M: \dist(x,y)<r_0\}
\]
and its complement $D^c$. Notice that $D^c$ is compact and is always a positive distance away from the diagonal, so $G$ is smooth on $D^c$ and therefore bounded there. We obtain:
\[
\int_{D^c} G(x,y)^2\,dx\,dy<\infty.
\]
Let us now analyze $D$. By \eqref{eq:GreenLogBound}, we set
\[
I(x)\coloneqq \int_{B_{r_0}(x)} \bigl(1+|\log \dist(x,y)|\bigr)^2\,dy,
\]
so that
\[
\int_{D} G(x,y)^2\,dx\,dy
\le C^2\int_{\M} I(x)\,dx.
\]
We also fix $x\in\M$. Note that on the ball $B_{r_0}(x)$, we may use geodesic polar coordinates centered at $x$: $y=\exp_x(\rho\theta)$ with $\rho\in(0,r_0)$ and $\theta\in\mathbb{S}^1$. In these coordinates, the volume form has the following form:
\[
\dd\mathrm{vol}(y)=J_x(\rho,\theta)\,\dd\rho\,\dd\theta,
\qquad
J_x(\rho,\theta)=\rho\,a_x(\rho,\theta),
\]
where $a_x$ is smooth and satisfies $a_x(0,\theta)=1$.
In particular, by compactness, there exists $C_1>0$ with:
\begin{equation}\label{eq:JacobianBound}
J_x(\rho,\theta)\le C_1\rho\qquad\text{for all }x\in\M,\ \theta\in\mathbb{S}^1,\ \rho\in(0,r_0)
\end{equation}
(see, for instance, \cite{Chavel84}).
We now apply the normalization of $dx$ and obtain:
\begin{align*}
\int_{B_{r_0}(x)} \bigl(1+|\log \dist(x,y)|\bigr)^2\,dy
&=\frac{1}{\mathrm{vol}(\M)}\int_{\mathbb{S}^1}\!\!\int_0^{r_0} \bigl(1+|\log \rho|\bigr)^2 J_x(\rho,\theta)\,\dd\rho\,\dd\theta\\
&\le \frac{C_1}{\mathrm{vol}(\M)}\int_{\mathbb{S}^1}\!\!\int_0^{r_0} \bigl(1+|\log \rho|\bigr)^2 \rho\,\dd\rho\,\dd\theta.
\end{align*}
The right-hand side is finite since near $0$ the function $\rho\mapsto \rho(\log \rho)^2$ is integrable. We apply the substitution $t=-\log\rho$ (so $\rho=e^{-t}$ and $\rho\,\dd\rho=-e^{-2t}\,\dd t$) and get:
\[
\int_0^{r_0} \bigl(1+|\log \rho|\bigr)^2\rho\,\dd\rho
=\int_{\infty}^{-\log r_0} \bigl(1+|t|\bigr)^2(-e^{-2t})\,\dd t
=\int_{-\log r_0}^{\infty} \bigl(1+|t|\bigr)^2 e^{-2t}\,\dd t <\infty.
\]
It follows that $\int_{B_{r_0}(x)} (1+|\log \dist(x,y)|)^2\,dy$ is bounded uniformly in $x$.
If we integrate over $x$, we arrive at $\int_D G^2\,dx\,dy<\infty$. Finally, we combine the estimates on $D$ and $D^c$. This finishes the proof of $\sigma^2<\infty$.
\end{proof}

\section{A second-moment bound for the Green energy}

Let $(X_i)_{i\ge 1}$ be i.i.d. $\M$-valued random variables with common law $dx$.
Let $\mu_n$ be the empirical measure defined above, with $x_i=X_i$, and set
\[
S_n \coloneqq \sum_{i\neq j} G(X_i,X_j).
\]
Since $G\in L^2(\M\times\M,dx\otimes dx)$ and $(dx\otimes dx)(\M\times\M)=1$, we have
\[
\|G\|_{L^1(\M\times\M,dx\otimes dx)}\le \|G\|_{L^2(\M\times\M,dx\otimes dx)},
\]
hence $G\in L^1(\M\times\M,dx\otimes dx)$ and $S_n\in L^1$.

\begin{lemma}\label{lem:mean0}
It holds that $\E[S_n]=0$.
\end{lemma}

\begin{proof}
We fix $i\neq j$. Since $G\in L^1(dx\otimes dx)$ is true, by the Fubini--Tonelli theorem,  we know that for $dx$-a.e. $x$ the section $y\mapsto G(x,y)$ is integrable, and the mean-zero normalization \eqref{eq:GreenMeanZero} gives
\[
\E\!\left[ G(X_i,X_j)\mid X_i\right]
=\int_{\M} G(X_i,y)\,dy=0\quad\text{a.s.}
\]
We take expectations and get $\E[G(X_i,X_j)]=0$.
Finally, it suffices to sum over all ordered pairs $(i,j)$ with $i\neq j$ to obtain $\E[S_n]=0$.
\end{proof}

\begin{lemma}\label{lem:secondmoment}
Let $\sigma^2$ be as in Lemma \ref{lem:L2Green}.
We have:
\[
\E[S_n^2]=2n(n-1)\sigma^2.
\]
It follows that
\[
\E|S_n| \le \sqrt{2n(n-1)}\,\sigma \le \sqrt{2}\,\sigma\, n.
\]
\end{lemma}

\begin{proof}
For $i\neq j$, we set $h_{ij}\coloneqq G(X_i,X_j)$.
Since $G\in L^2(dx\otimes dx)$ and $(X_i,X_j)\sim dx\otimes dx$ for $i\neq j$, we have $h_{ij}\in L^2$. For an arbitrary $(i,j)\neq(k,\ell)$ the product $h_{ij}h_{k\ell}\in L^1$ by Cauchy--Schwarz, and all expectations below are well defined.

We have
\[
S_n = \sum_{i\neq j} h_{ij},
\qquad
S_n^2 = \sum_{i\neq j}\sum_{k\neq \ell} h_{ij}h_{k\ell},
\qquad
\E[S_n^2] = \sum_{i\neq j}\sum_{k\neq \ell} \E[h_{ij}h_{k\ell}].
\]
We claim that $\E[h_{ij}h_{k\ell}]=0$ unless $(k,\ell)=(i,j)$ or $(k,\ell)=(j,i)$. Notice that if $\{i,j\}\cap\{k,\ell\}=\varnothing$, then $h_{ij}$ and $h_{k\ell}$ are independent. Moreover, each has mean $0$ by Lemma \ref{lem:mean0}, so $\E[h_{ij}h_{k\ell}]=0$. If $\{i,j\}$ and $\{k,\ell\}$ intersect in exactly one index, for example, for $i=k$ while $j\neq \ell$, then conditioning on $X_i$ and using conditional independence of $X_j$ and $X_\ell$ gives
\[
\E[h_{ij}h_{i\ell}\mid X_i]
=\E[h_{ij}\mid X_i]\;\E[h_{i\ell}\mid X_i]=0\cdot 0=0,
\]
again by \eqref{eq:GreenMeanZero}.
All other one-index overlap cases are identical. This follows directly from conditioning on the shared random variable and applying \eqref{eq:GreenMeanZero} in the correct slot. We conclude that only the two full-overlap cases contribute:
\begin{align*}
\E[h_{ij}^2]&=\E[G(X_1,X_2)^2]=\sigma^2,\\
\E[h_{ij}h_{ji}]&=\E[G(X_i,X_j)G(X_j,X_i)]=\E[G(X_1,X_2)^2]=\sigma^2
\end{align*}
(by symmetry $G(x,y)=G(y,x)$).
It follows that:
\[
\E[S_n^2]
=\sum_{i\neq j}\E[h_{ij}^2] + \sum_{i\neq j}\E[h_{ij}h_{ji}]
= n(n-1)\sigma^2+n(n-1)\sigma^2
=2n(n-1)\sigma^2.
\]
Finally, by Cauchy--Schwarz, it follows immediately that $\E|S_n|\le \sqrt{\E[S_n^2]}$.
\end{proof}

\begin{remark}[U-statistics]
Observe that Lemma \ref{lem:secondmoment} is a degenerate $U$-statistic case, that is, the kernel $G$ has mean zero, so by the Hoeffding decomposition, there is no linear term and all the cross terms also vanish.
\end{remark}

\section{Main result}

\begin{theorem}[The $\sqrt{\log n}$ remainder]\label{thm:main}
Let $G$ be the symmetric mean-zero Green function of $-\Delta$.
There does not exist a constant $C_{\M}>0$ such that for all $n\in\mathbb{N}$ and all $x_1,\dots,x_n\in \M$,
\begin{equation}\label{eq:NoUniversalO1sqrtN}
\W\!\left(\frac1n\sum_{i=1}^n\delta_{x_i},\, dx\right)
\le
C_{\M}\left(\frac{1}{\sqrt{n}}+\frac1n\left|\sum_{i\neq j} G(x_i,x_j)\right|^{1/2}\right).
\end{equation}
In particular, the $\sqrt{\log n}$ factor in the two-dimensional inequality \eqref{eq:Steinerberger2D} cannot be removed if the unrenormalized Green-energy term were to be preserved.
\end{theorem}

\begin{proof}
We assume for contradiction that there exists $C_{\M}>0$ such that \eqref{eq:NoUniversalO1sqrtN} holds for all $n$ and all point sets.
We apply it to the random configuration $X_1,\dots,X_n$.
Given that the inequality is deterministic and holds for all configurations, the following holds almost surely:
\[
\W(\mu_n,dx)\le C_{\M}\frac{1}{\sqrt{n}} + \frac{C_{\M}}{n}\,|S_n|^{1/2}.
\]
We now square and use $(a+b)^2\le 2a^2+2b^2$, which gives almost surely
\[
\W(\mu_n,dx)^2 \le \frac{2C_{\M}^2}{n} + \frac{2C_{\M}^2}{n^2}\,|S_n|.
\]
It suffices to take expectations (by using the uniform bound $\W(\mu_n,dx)\le \mathrm{diam}(\M)$ from Section 2.1) and apply Lemma \ref{lem:secondmoment},
\begin{equation}\label{eq:EW2upper}
\E\bigl[\W(\mu_n,dx)^2\bigr]
\le \frac{2C_{\M}^2}{n} + \frac{2C_{\M}^2}{n^2}\,\E|S_n|
\le \frac{2C_{\M}^2}{n}+\frac{2C_{\M}^2}{n^2}\cdot(\sqrt{2}\,\sigma\, n)
= \frac{C^\ast}{n},
\end{equation}
where
\[
C^\ast \coloneqq 2C_{\M}^2\bigl(1+\sqrt{2}\,\sigma\bigr)
\]
is a constant independent of $n$.

On the other hand, Ambrosio--Glaudo \cite{AG19} proved a useful two-sided asymptotic
for the semi-discrete matching problem on any compact closed surface. To match their
normalization, let $\tilde g = \mathrm{vol}(\M)^{-1} g$, so that $d\mathrm{vol}_{\tilde g}=dx$ and $\W$ scales by $\mathrm{vol}(\M)^{1/2}$. We apply \cite[Theorem 1.2]{AG19} to $(\M,\tilde g)$
and scale back, which implies that there exists a constant $C_{\mathrm{AG}}=C_{\mathrm{AG}}(\M)>0$ such that for
all integers $n\ge 3$,
\begin{equation}\label{eq:AGerror}
\left|\E\!\left[\W(\mu_n,dx)^2\right]-\frac{\mathrm{vol}(\M)}{4\pi}\,\frac{\log n}{n}\right|
\le
\frac{C_{\mathrm{AG}}}{n}\sqrt{\log n\,\log\log n}.
\end{equation}
See also \cite[Theorem 1.2]{AG19} and note the absolute value.

In particular, \eqref{eq:AGerror} implies the following lower bound
\begin{equation}\label{eq:AGlower}
\E\bigl[\W(\mu_n,dx)^2\bigr]
\ge
\frac{\mathrm{vol}(\M)}{4\pi}\,\frac{\log n}{n}
-\frac{C_{\mathrm{AG}}}{n}\sqrt{\log n\,\log\log n}.
\end{equation}
Since $\sqrt{\log n\,\log\log n}=o(\log n)$, there exists $n_1\ge 3$ such that
$C_{\mathrm{AG}}\sqrt{\log n\,\log\log n}\le \frac{\mathrm{vol}(\M)}{8\pi}\log n$ for all $n\ge n_1$.
This implies that \eqref{eq:AGlower} gives, for all $n\ge n_1$,
\begin{equation}\label{eq:AGlowerSimple}
\E\bigl[\W(\mu_n,dx)^2\bigr]\ge \frac{\mathrm{vol}(\M)}{8\pi}\,\frac{\log n}{n}.
\end{equation}

We now compare \eqref{eq:EW2upper} and \eqref{eq:AGlowerSimple}, and obtain for all $n\ge n_1$,
\[
\frac{\mathrm{vol}(\M)}{8\pi}\,\frac{\log n}{n}\le \E\bigl[\W(\mu_n,dx)^2\bigr]\le \frac{C^\ast}{n},
\]
so $\log n \le \frac{8\pi C^\ast}{\mathrm{vol}(\M)}$ for all $n\ge n_1$, which is impossible. This contradiction proves that no such constant $C_{\M}$ can exist.
\end{proof}

\begin{remark}[Obstruction]
In Theorem \ref{thm:main}, we rule out only universal inequalities of the exact form in Problem \ref{prob:Steinerberger}, that is, those with the unrenormalized off-diagonal Green term $\sum_{i\neq j}G(x_i,x_j)$. We emphasize that it does not preclude $n^{-1/2}$ rates for concrete deterministic point sets, nor does it preclude bounds with an appropriately renormalized Green energy.
\end{remark}

\end{document}